\documentclass[11pt]{article}  

\usepackage{amsmath, amsthm, amsfonts, amssymb}
\usepackage[bookmarks]{hyperref}
\usepackage{indentfirst} 
\usepackage{marvosym}

\usepackage[a4paper]{geometry}

\newtheorem{lemma}{Lemma}[section]
\newtheorem{theorem}[lemma]{Theorem}
\newtheorem{proposition}[lemma]{Proposition}
\newtheorem{corollary}[lemma]{Corollary}
\renewenvironment{proof}[1][\proofname]{{\sc #1. }}{\qed}
\newtheorem{theoremletters}{Theorem}

\theoremstyle{definition}

\newtheorem{example}[lemma]{Example}
{\bf}{\rm}

\newcommand{\abs}[1]{\ensuremath{\left| #1 \right|}}
\newcommand{\op}{\operatorname}
\newcommand{\ce}[2]{\pmb{\op{C}}_{#1}(#2)}
\newcommand{\no}[2]{\pmb{\op{N}}_{#1}(#2)}
\newcommand{\ze}[1]{\pmb{\op{Z}}(#1)}

\newcommand{\fit}[1]{\pmb{\op{F}}(#1)}
\newcommand{\rad}[2]{\pmb{\op{O}}_{#1}(#2)}
\newcommand{\radd}[2]{\pmb{\op{O}}_{#1, #1'}(#2)}
\newcommand{\syl}[2]{\op{Syl}_{#1}\left(#2\right)}
\newcommand{\hall}[2]{\op{Hall}_{#1}\left(#2\right)}
\newcommand{\van}[1]{{\op{Van}}(#1)}
\newcommand{\irr}[1]{{\op{Irr}}(#1)}

\begin{document}

\title{\bf On zeros of irreducible characters lying in a normal subgroup}

\author{\sc M. J. Felipe $^{*}$ $\cdot$ N. Grittini $^{\ddagger}$ $\cdot$ V. Sotomayor
\thanks{Instituto Universitario de Matemática Pura y Aplicada (IUMPA-UPV), Universitat Polit\`ecnica de Val\`encia, Camino de Vera s/n, 46022 Valencia, Spain. \newline
\indent $^{\ddagger}$Dipartimento di Matematica U. Dini, Universit\`a degli Studi di Firenze, viale Morgagni 67/a, 50134 Firenze, Italy.\newline
\Letter: \texttt{mfelipe@mat.upv.es}, \texttt{nicola.grittini@unifi.it}, \texttt{vicorso@doctor.upv.es} \newline \rule{6cm}{0.1mm}\newline
The first author is supported by Proyecto Prometeo II/2015/011, Generalitat Valenciana (Spain). The research of the second author is partially funded by the Istituto Nazionale di Alta Matematica - INdAM. The third author acknowledges the predoctoral grant ACIF/2016/170, Generalitat Valenciana (Spain). The first and third authors are also supported by Proyecto PGC2018-096872-B-I00, Ministerio de Ciencia, Innovación y Universidades (Spain). \newline
}}

\date{}

\maketitle

\begin{abstract}
\noindent Let $N$ be a normal subgroup of a finite group $G$. In this paper, we consider the elements $g$ of $N$ such that $\chi(g)\neq 0$ for all irreducible characters $\chi$ of $G$. Such an element is said to be \emph{non-vanishing in} $G$. Let $p$ be a prime. If all $p$-elements of $N$ satisfy the previous property, then we prove that $N$ has a normal Sylow $p$-subgroup. As a consequence, we also study certain arithmetical properties of the $G$-conjugacy class sizes of the elements of $N$ which are zeros of some irreducible character of $G$. In particular, if $N=G$, then new contributions are obtained.

\medskip


\smallskip

\end{abstract}


\section{Introduction}

In the sequel, all groups considered are finite. Within character theory, a classical theorem of Burnside asserts that a non-linear irreducible character of a finite group always vanishes on some element. It is not difficult to see that the converse is also true, so the rows of the character table of a group that contain a zero entry are completely characterised. However, the ``dual'' situation for conjugacy classes fails in general: a column that corresponds to a non-central conjugacy class may not contain a zero. This fact somehow violates the standard duality that in many cases arises between irreducible characters and conjugacy classes of a group. Therefore, for a group $G$, an element $g$ is said to be \emph{non-vanishing} in $G$ if $\chi(g)\neq 0$ for every irreducible character $\chi$ of $G$. 

An immediate corollary to the aforementioned Burnside's result is that a group is abelian if and only if every element is non-vanishing. I.M. Isaacs, G. Navarro and T.R. Wolf obtained in \cite{INW} elegant results about the location of non-vanishing elements in certain groups. For example, for a nilpotent group $G$, an element is non-vanishing if and only if it lies in the centre of $G$. They also proved that if $G$ is soluble, then $g\fit{G}$ is a $2$-element for a non-vanishing element $g$ of $G$. Consequently, if $g$ is of odd order, then $x$ lies in $\fit{G}$. These authors conjectured that every non-vanishing element of a
soluble group $G$ lies in $\fit{G}$, and it is still an open problem. In this paper, we prove the following result which provides further evidence for this conjecture.

\begin{theoremletters}
\label{teoA}
Let $N$ be a normal subgroup of a group $G$, and let $p$ be a prime. If $\chi(x)\neq 0$ for every $p$-element $x\in N$ and for all $\chi\in\op{Irr}(G)$, then $N$ has a normal Sylow $p$-subgroup. 

In particular, if $\chi(x)\neq 0$ for every prime power order element $x\in N$ and for all $\chi\in\op{Irr}(G)$, then $N$ is nilpotent.
\end{theoremletters}

Therefore, the arithmetical properties of the non-vanishing elements of $G$ that lie in a normal subgroup $N$ control the structure of $N$. This is interesting since, although we cannot construct the character table of $N$ from the one of $G$, normal subgroups and vanishing elements of $G$ can be easily read from its character table.

Regarding the first claim of Theorem \ref{teoA} when $N=G$, we provide extra information on the structure of a $p$-complement of $G$ in Corollary \ref{cor-dpss}, which extends \cite[Theorem A]{DPSS}. Concerning the second assertion in Theorem \ref{teoA}, when $N=G$ it holds that the group is abelian (see Theorem \ref{carac}). However, this fact might not happen for the case of a normal subgroup as Example \ref{suz} shows.

Actually, we prove Theorem \ref{teoA} as a consequence of the next result. We will denote by $1_G$ the trivial character of a group $G$.

\begin{theoremletters}
\label{TEOB}
Let $N$ be a normal subgroup of a group $G$, and let $P$ be a Sylow $p$-subgroup of $G$ for some prime $p$. Let $P_0=P\cap N$ and $\beta\in\irr{P/P_0}$. Then the following conditions are pairwise equivalent:

\emph{(i)} $P_0$ is a normal Sylow $p$-subgroup of $N$.

\emph{(ii)} $\chi(x)\neq 0$ for all irreducible constituents $\chi$ of $(1_{P_0})^G$ and all $x\in P_0$.

\emph{(iii)} $\chi(x)\neq 0$ for all irreducible constituents $\chi$ of $\beta^G$ and all $x\in P_0$.
\end{theoremletters}

Indeed Theorem \ref{TEOB} generalises \cite[Theorem B]{MN} when $N=G$ (see Theorem \ref{MN}). Notice that, by Theorem \ref{MN} (i)-(ii), $P_0$ is normal in $N$ if and only if $\eta(x)\neq 0$ for all irreducible constituents $\eta$ of $(1_{P_0})^N$ and all $x\in P_0$; however this fact does not directly imply (ii) of Theorem \ref{TEOB}, nor vice versa. Further, the following equivalence, which is related to Theorem \ref{MN} (i)-(iii), is not true: $P_0$ is normal in $N$ if and only if $p$ does not divide $\chi(1)$ for all irreducible constituents of $(1_{P_0})^G$; it is enough to observe Example \ref{degree} (2).

As a consequence of Theorem \ref{teoA}, some features of a normal subgroup $N$ of a group $G$ can be obtained through the analysis of its $G$-conjugacy class sizes of elements which are zeros of some irreducible character of $G$; such an element is said to be \emph{vanishing} in $G$. Let $\pi$ be any set of primes, and let $x^G$ be the conjugacy class of $x$ in $G$. If $\abs{x^G}$ is a $\pi'$-number for every prime power order $\pi$-element $x$ in $N$, then by the main result of A. Beltrán et al. (see \cite[Theorem B]{BFM}), it is known that $N$ has a nilpotent Hall $\pi$-subgroup. In the next result, we show that we do not need to assume this condition for all prime power order $\pi$-elements in $N$, but for those which are vanishing in $G$. However, we need to assume the $\pi$-separabilty of $N$ to get that result.

\begin{theoremletters}
\label{teoC}
Let $N$ be a normal subgroup of a group $G$, and let $\pi$ be any set of prime numbers. 

\begin{enumerate}
	\item[\emph{(1)}] Suppose that $\abs{x^G}$ is a $\pi'$-number for every prime power order $\pi$-element $x \in N$ which is vanishing in $G$. If $N$ is $\pi$-separable, then $N/\rad{\pi'}{N}$ has a nilpotent normal Hall $\pi$-subgroup. In particular, the Hall $\pi$-subgroups of $N$ are nilpotent.
	
	\item[\emph{(2)}] Suppose that $\abs{x^G}$ is a $\pi$-number for every prime power order $\pi$-element $x \in N$ which is vanishing in $G$. If $\hall{\pi}{N}\neq \emptyset$, then $N$ has a normal Hall $\pi$-subgroup. 
	
Additionally, if all $\abs{x^G}$ are also $\pi$-numbers for the prime power order $\pi'$-elements $x\in N$ that are vanishing in $G$, then the Hall $\pi'$-subgroups of $N$ are nilpotent.
\end{enumerate}
\end{theoremletters}

We do not know whether the $\pi$-separability condition on $N$ in (1) can be weakened in order to obtain the nilpotency of its Hall $\pi$-subgroups. What is certainly true is that this condition is necessary for the normality of the Hall $\pi$-subgroup of $N/\rad{\pi'}{N}$, as Example \ref{pi-separa} shows. Additionally, the statement (2) above extends for a set of primes the following result proved in \cite{BQ}: If a prime $p$ does not divide any conjugacy class size of a vanishing $p'$-element $x$ of prime power order of a group $G$, then $G$ has a normal $p$-complement. We do not know whether the assumption $\hall{\pi}{N}\neq \emptyset$ in (2) can be avoided.

Finally, we investigate in Theorem \ref{teoE} the structure of $N$ when the $G$-conjugacy class lenghts of the considered vanishing elements in $G$ are prime powers. As a consequence of this study, when $N=G$ we obtain the next result.

\begin{theoremletters}
\label{cor_nilpotent}
Let $G$ be a group. Assume that $\abs{x^G}$ is a prime power for every vanishing element $x$ of $G$ of prime power order. Then $G'$ is nilpotent.
\end{theoremletters}

Other new interesting consequences arise from our contributions in the trivial case $N=G$ (see Section \ref{secCor}).


\section{Preliminaries}

The notation and terminology here is as follows. In the sequel, $p$ will be always a prime, and $\pi$ will denote a set of primes. The set of prime divisors of the order of $G$ is $\pi(G)$. As usual, the set of all Sylow $p$-subgroups of $G$ is denoted by $\syl{p}{G}$, and $\hall{\pi}{G}$ is the set of all Hall $\pi$-subgroups of $G$. We write $\op{Irr}(G)$ for the set of all irreducible complex characters of $G$. The set of vanishing elements of a group $G$ will be denoted by $\van{G}$. CFSG means classification of finite simple groups. The remaining notation and terminology is standard in the framework of finite group theory, and we refer to the book \cite{ISA} for details about character theory.

We gather some significant results for locating vanishing elements in a given group. As mentioned in the Introduction, it is elementary to see that a group is abelian if and only if every element is non-vanishing. In fact, this characterisation remains true, via the CFSG, when only prime power order elements are involved.

\begin{theorem}
\label{carac}
A group $G$ is abelian if and only if every prime power order element is non-vanishing in $G$.
\end{theorem}

\begin{proof}
This is a direct application of \cite[Theorem B]{MNO}, which asserts that a non-linear irreducible character vanishes on some prime power order element.
\end{proof}

\begin{example}
\label{suz}
Concerning the above theorem it is worth noting that, in general, it is not true that a normal subgroup $N$ is abelian if and only if every element of $N$ is non-vanishing in $G$, i.e. if $N\cap \van{G}=\emptyset$:

On the one hand, if $G=Q_8$ is a quaternion group of $8$ elements and $N$ is a normal subgroup of $G$ isomorphic to a cyclic group of order $4$, then $N$ is abelian and $N\cap\van{G}\neq \emptyset$. On the other hand, by \cite[Theorem 5.1]{INW}, for any prime $p$ there exists a group $G$ having a normal non-abelian Sylow $p$-subgroup, and every $p$-element of $G$ is non-vanishing.
\end{example}

\begin{proposition}\emph{\cite[Theorem B]{INW}}
\label{nilp}
$G\smallsetminus\ze{G}= \van{G}$ for any nilpotent group $G$.
\end{proposition}

Observe that if a normal subgroup $N$ is nilpotent, then $N\smallsetminus\ze{G}$ may not coincide with $\van{G}\cap N$. For instance, one can consider as $G$ the normaliser in a Suzuki group of degree 8 of a Sylow $2$-subgroup of it, and $N$ the Sylow $2$-subgroup. It holds that $\van{G}\cap N=\emptyset$ although clearly $N\smallsetminus\ze{G}\neq \emptyset$.

\begin{lemma}\emph{\cite[Corollary 1.3]{G}}
\label{gru}
Let $H$ be a subgroup of a group $G$. Assume that $G=H\ce{G}{x}$ for some $x\in H$. Then $x\in \van{G}$ if and only if $x\in \van{H}$.
\end{lemma}

The four lemmas below are crucial for proving Theorem \ref{teoB}, and the last two use the CFSG.

\begin{lemma}\emph{\cite[Lemma 5]{BCLP}}
\label{BianchiLemma}
Let $N$ be a minimal normal subgroup of $G$ so that $N = S_1 \times \cdots \times S_t$, where $S_i$ is isomorphic to $S$, a non-abelian simple group. If $\sigma\in\irr{S}$ extends to $\op{Aut}(S)$, then $\sigma \times \cdots \times \sigma \in \irr{N}$ extends to $G$.
\end{lemma}

\begin{lemma}\emph{\cite[Lemma 2.2]{MN}}
\label{2.2MN}
Let $G$ be a finite group, $p$ a prime, and $P\in\syl{p}{G}$. If $\chi\in\irr{G}$ has $p$-defect zero, then $\chi$ is a constituent of $(1_P)^G$ and vanishes on the non-trivial $p$-elements of $G$.
\end{lemma}

\begin{lemma}\emph{\cite[Theorem 2.1]{MN}}
\label{2.1MN}
Let $S$ be a finite non-abelian simple group, $p$ a prime, and $P\in\syl{p}{S}$. Then either $S$ has a
$p$-defect zero character, or there exists a cons\-tituent $\theta \in \irr{S}$ of the permutation character $(1_P)^S$ such that $\theta$ extends to $\op{Aut}(S)$ and $\theta(x)=0$ for some $p$-element $x$ of $S$.
\end{lemma}

\begin{lemma}\emph{\cite[Lemma 2.8]{DPSS}}
\label{contradiction_lemma}
Let $A$ be an abelian group that acts coprimely and faithfully by automorphisms on a group $M$. If $M$ is characteristically simple, then there exists $\theta \in \op{Irr}(M)$ such that $I_A(\theta)=1$.
\end{lemma}

We also collect some preliminary results regarding conjugacy class sizes. We start with the next elementary properties which are frequently used, sometimes with no comment.

\begin{lemma}
Let $N$ be a normal subgroup of a group $G$, and let $p$ be a prime. We have:

\emph{(a)} $\abs{x^N}$ divides $\abs{x^G}$, for any $x\in N$.
	
\emph{(b)} $\abs{(xN)^{G/N}}$ divides $\abs{x^G}$, for any $x\in G$.
	
\emph{(c)} If $xN\in G/N$ is a $p$-element, then $xN=yN$ for some $p$-element $y\in G$.

\end{lemma}

\begin{lemma}
\label{wielandt}
Let $N$ be a normal subgroup of a group $G$, and let $H\in\hall{\pi}{N}$ for a set of primes $\pi$. If $x\in H$ is such that $\abs{x^G}$ is a $\pi$-number, then $x$ lies in $\rad{\pi}{N}$.
\end{lemma}

\begin{proof}
Since $\abs{x^N}$ divides $\abs{x^G}$, then $(\abs{x^N}, \abs{N:H})=1$. It follows $N=H\ce{N}{x}$ and so $\langle x^N\rangle \leqslant \rad{\pi}{N}$.
\end{proof}

\bigskip

Next we recall a generalisation of the above lemma when $N=G$ and $\pi=\{p\}$.

\begin{lemma}\emph{\cite[Lemma 3]{BK}}
\label{berkokazarin}
Let $x\in G$. If $\abs{x^G}$ is a power of a prime $p$, then $[x^G, x^G]$ is a $p$-group.
\end{lemma}

We end this section with the main result of \cite{Br}, which will be necessary for proving Theorem \ref{teoE}. We present here an adapted version for our context of vanishing $G$-conjugacy classes.

\begin{proposition}
\label{brough}
Let $G$ be a group which contains a non-trivial normal $p$-subgroup $N$, for a given prime $p$. Then  $\abs{x^G}$ is a multiple of $p$ for each $x\in N\cap \van{G}$.
\end{proposition}


\section{Proof of Theorems \ref{teoA} and \ref{TEOB}} 

Certainly, Theorem \ref{teoA} is a direct application of Theorem \ref{TEOB}, so we focus on the proof of this last result. The next key proposition, which makes use of the CFSG, is inspired by the proof of \cite[Theorem B]{MN}.

\begin{proposition}
\label{lemma:minimalsubgroup}
Let $M$ be a non-abelian minimal normal subgroup of a group $G$, and let $p$ be a prime divisor of $\abs{M}$. Let $H$ be a subgroup of $G$ such that $H \cap M\in\syl{p}{M}$. Let $\beta\in\irr{H/H\cap M}$ Then, there exists $\chi \in \op{Irr}(G)$ such that $\chi$ is a constituent of $\beta^G$ and it vanishes on some $p$-element of $M$.

In particular, if $H=P\in\syl{p}{G}$, then there exists $\chi\in \op{Irr}(G)$ such that $\chi$ is a constituent of $(1_P)^G$ and it vanishes on some $p$-element of $M$.
\end{proposition}

\begin{proof}
We have  $M=S_1 \times \cdots \times S_k$, where all $S_i$ are isomorphic to a non-abelian simple group $S$ with $p\in\pi(S)$.  If $\theta\in\irr{S}$ is of $p$-defect zero, then $\eta=\theta \times \cdots \times \theta\in\irr{M}$ and $\eta$ is also of $p$-defect zero. By Lemma \ref{2.2MN} applied to $M$ we have $[\eta, (1_{H\cap M})^M]\neq 0$ and $\eta$ vanishes on the non-trivial $p$-elements of $M$. 

Since $\beta\in\irr{H/H\cap M}$, we have  $[\beta_{H\cap M}, 1_{H\cap M}]\neq 0$. Then $(\beta^{HM})_M=(\beta_{H\cap M})^M=\beta(1)(1_{H\cap M})^M$ and $[\eta, (\beta^{HM})_M]=[\eta^{HM}, \beta^{HM}]\neq 0$. Hence there exists $\tau\in\irr{HM}$ such that $[\tau, \eta^{HM}]\neq 0 \neq [\tau, \beta^{HM}]$. Let $\chi\in\irr{G}$ over $\tau$. Then $\chi_M$ is sum of $G$-conjugate characters of $\eta$. Therefore $\chi$ vanishes on the non-trivial $p$-elements of $M$ and $[\chi, \beta^G]=[\chi_H, \beta]\neq 0$.

Suppose now that $S$ does not have a character of $p$-defect zero. By Lemma \ref{2.1MN}, there exists $\theta\in\irr{S}$ such that $[\theta, (1_{H\cap S})^S]\neq 0$ (note $H\cap S\in\syl{p}{S}$) which extends to $\op{Aut}(S)$, and there exists a $p$-element $x\in S$ such that $\theta(x)=0$. Thus $1\neq y=(x, \ldots, x)\in M$ is a $p$-element and $\eta=\theta\times\cdots\times\theta$ vanishes on $y$, and certainly $[\eta_{H\cap M}, 1_{H\cap M}]\neq 0$. Since $[\beta_{H\cap M}, 1_{H\cap M}]\neq 0$, arguing as in the previous paragraph, we may affirm that there exists $\tau \in\irr{HM}$ over $\eta$ and over $\beta$. Let $\chi\in\irr{G}$ be over $\eta$, so $[\chi, \beta^G]\neq 0$. By Lemma \ref{BianchiLemma}, $\eta$ extends to $G$. Let $\hat{\eta}$ be an extension of $\eta$. By Gallagher, $\chi=\hat{\eta}\rho$ for some $\rho\in\irr{G/M}$. Therefore, $\chi$ lies over $\beta$ and $\chi(y)=\eta(y)\rho(1)=0$.
\end{proof}

\begin{theorem}
\label{teoB}
Let $N$ be a normal subgroup of a group $G$, and let $P_0$ be a Sylow $p$-subgroup of $N$ for some prime $p$. Let $H$ be a subgroup of $G$ such that $H\cap N=P_0$, and let $\beta \in \irr{H/P_0}$. Then, $P_0$ is normal in $N$ (and therefore in $G$) if and only if all irreducible constituents of $\beta^G$ do not vanish on any $p$-element of $N$.
\end{theorem}

\begin{proof}
Suppose $P_0\unlhd N$. Let $\chi$ be a constituent of $\beta^G$ with $\beta\in\irr{H/H\cap N}$.  We have $[\beta_{P_0}, 1_{P_0}]\neq 0$, so $[\chi_{P_0},1_{P_0}]\neq 0$. Since $P_0\unlhd G$, then $\chi(x)\neq 0$ for all $p$-elements $x\in N$. Conversely, we consider that all irreducible constituents of $\beta^G$, where $\beta\in\irr{H/H\cap N}$, do not vanish on any $p$-element of $N$, and we claim that $P_0$ is normal in $N$. 

Suppose that the claim is false, and let us consider a counterexample which minimises $\abs{G}$. Let $M$ be a minimal normal subgroup of $G$ such that $M\leqslant N$. We check that the hypotheses are inherited by $\overline{G}=G/M$. Certainly $\overline{H}\cap \overline{N}=N/M\cap HM/M=(H\cap N)M/M\in\syl{p}{N/M}$. Since $\beta\in\irr{H/H\cap N}$, then $\beta \in\irr{H/H\cap M}$ so $\overline{\beta}\in\irr{HM/M}$. Besides $H\cap N\leqslant \op{ker}{\beta}$ so $\overline{H\cap N}\leqslant \op{ker}{\overline{\beta}}$. Let $\overline{\chi}\in\irr{\overline{G}}$ be an irreducible constituent of $\overline{\beta}^{\overline{G}}$ and $\overline{x}\in\overline{N}$ a $p$-element. Then we may assume that $x\in N\smallsetminus M$ is a $p$-element and, since $[\overline{\chi}, \overline{\beta}^{\overline{G}}]\neq 0$, then it is easy to see that $[\chi_H, \beta]\neq 0$ and $\overline{\chi}(\overline{x})=\chi(x)\neq 0$. By minimality, we get $\overline{P_0}\unlhd \overline{G}$, so $P_0M\unlhd G$.

Let us assume that $p$ divides the order of $M$. If $M$ is a $p$-group, then $M\leqslant P_0$ and $P_0=P_0M\unlhd G$, a contradiction. Hence $M$ is non-abelian. Since $\beta\in\irr{H/H\cap M}$, in virtue of Lemma \ref{lemma:minimalsubgroup} there exists $\chi\in\irr{G}$ such that $[\chi, \beta^G]\neq 0$ and $\chi(x)=0$ for some $p$-element $x\in M\leqslant N$, a contradiction again.

Thus $p$ does not divide the order of $M$ and $\rad{p}{N}=1$. Let $K/M$ be a chief factor of $G$ such that $K \leq P_0M\unlhd G$, so $K/M$ is an abelian $p$-group. Note $K=M(K\cap P_0)$ and $K\cap P_0\in\syl{p}{K}$ is abelian. By Frattini's argument, $G=K\no{G}{K\cap P_0}=M\no{G}{K\cap P_0}$, so $\ce{K\cap P_0}{M}\unlhd G$ and $\ce{K\cap P_0}{M}\leqslant\rad{p}{N}=1$. Therefore $K\cap P_0$ is an abelian $p$-group which acts coprimely and faithfully on $M$, and $M$ is characteristically simple. By Lemma \ref{contradiction_lemma} and Clifford theory, there exists $\theta \in \op{Irr}(M)$ such that $\eta =\theta^K$ is irreducible. In particular, $\eta$ and all its conjugates vanish on $K \setminus M$. Therefore, if we prove that there exists $\chi \in \op{Irr}(G)$ which lies over both $\eta$ and $\beta$ we will reach the final contradiction.

Let $T$ be the inertia subgroup for $\theta$ in $P_0M\unlhd G$. Since $(\abs{T/M},\abs{M})=1$ we have that $\theta$ extends to $\hat{\theta} \in \op{Irr}(T)$ by \cite[Corollary 6.28]{ISA}. Further, $p$ does not divide $\hat{\theta}(1)$ so $\hat{\theta}_{P_0 \cap T}$ has at least one linear constituent $\lambda$. As $T=M(P_0\cap T)$, then $P_0 \cap T \cong T/M$ and we can see $\lambda$ also as a character of $T/M$. By Gallagher, $\nu = \bar{\lambda}\hat{\theta}$ is an irreducible character of $T$, where $\bar{\lambda}$ is the complex conjugate of $\lambda$. Moreover, $\nu_M=\theta$ and by Clifford correspondence $\nu^{P_0M}\in\irr{P_0M}$. Hence $0\neq [1_{P_0\cap T}, \overline{\lambda}_{P_0\cap T}\hat{\theta}_{P_0\cap T}]=[1_{P_0\cap T}, \nu_{P_0\cap T}]=[(\nu_{P_0\cap T})^{P_0}, 1_{P_0}]=[(\nu^{P_0T})_{P_0}, 1_{P_0}]=[(\nu^{P_0M})_{P_0}, 1_{P_0}]=[\nu^{P_0M}, (1_{P_0})^{P_0M}]$. On the other hand, $(\beta^{HN})_N = \beta(1)(1_{P_0})^N=\beta(1)((1_{P_0})^{P_0M})^N$, so $[(\beta^{HN})_N, (\nu^{P_0M})^N]=[(\beta^{HN})_N, \nu^N]=[\beta^{HN}, \nu^{HN}]\neq 0$. Therefore there exists $\tau\in\irr{HN}$ over $\beta$ and over $\nu$. Let $\chi\in\irr{G}$ over $\tau$, so $[\chi, \beta^G]\neq 0$. Moreover, $\chi$ lies over $\theta$, and then $\chi$ lies over $\eta =\hat{\theta}$. Thus $\chi_K$ is a sum of $G$-conjugate characters of $\eta$. Hence $\chi(x)=0$ for all $x\in K\cap P_0$ and this is a final contradiction.
\end{proof}

\bigskip

Theorem \ref{TEOB} in the Introduction is now a corollary of the above result when we take $H$ a Sylow $p$-subgroup of $G$ (for Theorem \ref{TEOB} (iii)) and $H=P_0$ (for Theorem \ref{TEOB} (ii)). Moreover, when $N=G$ in Theorem \ref{TEOB}, then we obtain the characterisation (i)-(ii) below.

\begin{theorem}\emph{\cite[Theorem B]{MN}}
\label{MN}
Let $G$ be a group, $p$ a prime number, and $P$ a Sylow $p$-subgroup of $G$. Then the following conditions are equivalent:

\emph{(i)} $P$ is normal in $G$.

\emph{(ii)} $\chi(x)\neq 0$ for all irreducible constituents $\chi$ of $(1_P)^G$ and all $x\in P$.

\emph{(iii)} $p$ does not divide $\chi(1)$ for all irreducible constituents $\chi$ of $(1_P)^G$.
\end{theorem}

\smallskip

\begin{example}
\label{degree}
(1) Note that in Theorem \ref{TEOB} we can have $\beta\in\irr{P/P_0}$ distinct from $1_{P}$, in contrast to Theorem \ref{MN}: Let $G$ be a symmetric group of degree $4$ and let $N$ be an alternating group of degree $4$. Take $P\in\syl{2}{G}$. Then there exists a non-trivial irreducible character $ \beta\in\irr{P}$ with $P_0=P\cap N\leqslant \op{ker}{\beta}$. Additionally, the irreducible constituents of $\beta^G$ do not vanish on the $p$-elements of $N$, so the hypotheses in Theorem \ref{TEOB} (iii) are fulfilled.

(2) The following equivalence, similar to Theorem \ref{MN} (i)-(iii), is not true: $P_0$ is a normal Sylow $p$-subgroup of $N$ if and only if $p$ does not divide $\chi(1)$ for all irreducible constituents of $(1_{P_0})^G$: Consider $G$ and $N$ as above. Then $(1_{P_0})^G$ has three distinct irreducible constituents, being one of them of degree $2$.

Both examples have been checked using the software \texttt{GAP} \cite{GAP}.
\end{example}

Let consider now a set of primes $\pi$ instead of a single prime $p$. As a consequence of Theorem \ref{teoA}, we give in the following proposition extra information on the structure of a $\pi$-complement of $G$ when $N$ contains a Hall $\pi$-subgroup of it.

\begin{proposition}
\label{nilphall}
Let $N$ be a normal subgroup of a group $G$ such that every prime power order $\pi$-element of $N$ is non-vanishing in $G$, for a set of primes $\pi$. Then $N$ has a nilpotent normal Hall $\pi$-subgroup.

Further, if $\abs{G:N}$ is a $\pi'$-number, then any $\pi$-complement $F$ of $G$ verifies that $F\ze{G}$ is self-normalising.
\end{proposition}

\begin{proof}
Certainly, in virtue of Theorem \ref{teoA} we have that $N$ has a nilpotent normal Hall $\pi$-subgroup, say $H$. In fact, if $\abs{G:N}$ is not divisible by any prime in $\pi$, then $H$ is a normal Hall $\pi$-subgroup of $G$. Let $F$ be a $\pi$-complement of $H$ in $G$, so $G=HF$. We aim to show that $F\ze{G}=\no{G}{F\ze{G}}$. Take a prime power order element $x\in\no{H}{F\ze{G}}$. Then $F^x\ze{G}=(F\ze{G})^x=F\ze{G}$, so there exists some $y\in F\ze{G}$ such that $F^x=F^y=F$. Thus $x\in \no{H}{F}\leqslant \ce{H}{F}$ because $[\no{H}{F}, F]\leqslant H\cap F=1$. Therefore $G=HF=H\ce{G}{x}$. Since $x\notin \van{G}$ by assumption, then Lemma \ref{gru} yields that $x\notin \van{H}$. Now Proposition \ref{nilp} applies because $H$ is nilpotent, so $x\in \ze{H}\cap\ce{G}{F}\leqslant\ze{G}$. As this argument is valid for every prime power order element in $\no{H}{F\ze{G}}$, then $\no{H}{F\ze{G}}\leqslant \ze{G}$. Finally, note that $\no{G}{F\ze{G}}= \no{G}{F\ze{G}} \cap HF = F(\no{H}{F\ze{G}})=F\ze{G}$, as wanted.
\end{proof}

\begin{corollary}
\label{cor-dpss}
Let $G$ be a group such that all the $p$-elements are non-vanishing. Then $G$ has a normal Sylow $p$-subgroup, and $F\ze{G}$ is self-normalising for any $p$-complement $F$ of $G$.
\end{corollary}

\section{Lenghts of \texorpdfstring{$G$}{G}-conjugacy classes of vanishing elements}

We start by showing an extension of Lemma \ref{berkokazarin} for a set of primes $\pi$ and a $G$-conjugacy class. The proof is inspired by \cite[Theorem C]{BF} under the weaker hypothesis of the $\pi$-separability of the normal subgroup $N$.

\begin{proposition}
\label{in_fitting}
Let $N$ be a normal $\pi$-separable subgroup of a group $G$. If $x \in N$ is such that $\abs{x^G}$ is a $\pi$-number, then $[x^G,x^G] \leqslant \rad{\pi}{N}$. In particular, $x\rad{\pi}{N}/\rad{\pi}{N}\in\ze{\fit{N/\rad{\pi}{N}}}$.

Indeed, if $\pi$ consists of a single prime $p$, then the same statement is valid even if $N$ is not $p$-soluble.
\end{proposition}

\begin{proof}
In order to prove the first claim, let us consider a counterexample which mini\-mises $\abs{G}+\abs{N}$. One can clearly assume $\rad{\pi}{N}=1$, so we aim to get the contradiction $[x^G,x^G]=1$. Let us suppose firstly that $\langle x \rangle$ is subnormal in $G$. Then $x \in \fit{G}$. As $\fit{G}$ is a $\pi'$-group and $\abs{x^G}$ is a $\pi$-number, then clearly $x \in \ze{\fit{G}}$ and $\langle x^G\rangle \leqslant \ze{\fit{G}}$, so $[x^G,x^G]=1$.

Next we assume that the normal subgroup $M:=\langle x^G\rangle$ is proper in $N$. Then by minimality we obtain $[x^M,x^M]=1$, and it follows that $x \in \ze{\langle x^M \rangle}$. In particular, $\langle x \rangle$ is subnormal in $M$, and therefore in $G$, which contradicts the previous paragraph. Hence $M=N$. 

Let $K:=\rad{\pi'}{N}$. Since $N$ is $\pi$-separable, then $K$ is non-trivial. It follows from the class size hypothesis that $K$ centralises $x^G$, so $K$ is central in $N=\langle x^G\rangle$. As $[x^G,x^G]K/K \leqslant \rad{\pi}{N/K}$ by minimality, and $\rad{\pi}{N/K}=\rad{\pi}{N}K/K$ because $K$ is central $N$, we deduce $[x^G,x^G]=N' \leqslant K\leqslant \ze{N}$. Therefore $N$ is a nilpotent $\pi'$-group. Since $\abs{x^G}$ is a $\pi$-number, we obtain $x\in \ze{N}$ and $[x^G,x^G]=1$.

Next we concentrate on the second assertion. Let $\overline{G}:=G/\rad{\pi}{N}$. Then, $[\overline{x}^{\overline{G}},\overline{x}^{\overline{G}}]=1$ by the first claim. It follows that $\langle \overline{x}\rangle \unlhd \ze{\langle \overline{x}^{\overline{G}} \rangle} \unlhd \overline{G}$, so $\langle \overline{x}\rangle\leqslant \fit{\overline{G}} \cap \overline{N} \leqslant \fit{\overline{N}}$. As $\fit{\overline{N}}$ is a normal $\pi'$-subgroup of $\overline{G}$ and $|\overline{x}^{\overline{G}}|$ is a $\pi$-number, then necessarily $\overline{x}\in \ze{\fit{\overline{N}}}$.

Finally, observe that the last statement follows from Lemma \ref{berkokazarin}, since $[x^G, x^G]\leqslant \rad{p}{G}\cap N\leqslant \rad{p}{N}$.
\end{proof}

\begin{example}
Note that the $\pi$-separability assumption in the previous result cannot be removed, even when $N=G$: Consider any non-trivial element in the centre of a Sylow $p$-subgroup of a non-abelian simple group and $\pi=p'$, for a prime divisor $p$ of its order.
\end{example}

For a normal subgroup $N$ of a group $G$, note that if $xN$ is a vanishing (prime power order) element of $G/N$, then we can assume that $x$ is also a vanishing (prime power order) element of $G$. This is because there exists a bijection between $\op{Irr}(G/N)$ and the set of all characters in $\op{Irr}(G)$ containing $N$ in their kernel. This fact will be used in the sequel with no reference.

As an application of the above proposition and mainly Theorem \ref{teoA}, we prove Theorem \ref{teoC} in the Introduction.

\bigskip

\begin{proof}[Proof of Theorem \ref{teoC}]
(1) Assume that $N$ is $\pi$-separable, and that $\abs{x^G}$ is a $\pi'$-number for every prime power order $\pi$-element $x \in \van{G} \cap N$. Let us prove that $N/\rad{\pi'}{N}$ has a normal Sylow $p$-subgroup for each prime $p\in\pi$. Certainly, whenever $\rad{\pi'}{N}\neq 1$, the assertion follows by induction, considering the groups $G/\rad{\pi'}{N}$ and $N/\rad{\pi'}{N}$. Therefore we may assume that $\rad{\pi'}{N}=1$. Let $Z_p:=\ze{\rad{p}{N}}$. In virtue of Proposition \ref{in_fitting}, we have that all the $p$-elements of $\van{G} \cap N$ lie in $\ze{\fit{N}}$, and thus in $Z_p$. Therefore, if we denote $\overline{G}:=G/Z_p$, then it follows that no prime power order $p$-element of $\overline{N}$ is vanishing in $\overline{G}$. Now Theorem \ref{teoA} yields that $\overline{N}$ has a normal Sylow $p$-subgroup $\overline{P}$, where $P\in\syl{p}{N}$. Since $Z_p$ is a $p$-group, then $P$ is normal in $N$ clearly and we get the claim. As this is valid for each prime $p\in \pi$, then $N/\rad{\pi'}{N}$ has a nilpotent normal Hall $\pi$-subgroup, as wanted.

(2) Assume that $N$ has Hall $\pi$-subgroups, and that $\abs{x^G}$ is a $\pi$-number for every prime power order $\pi$-element $x \in \van{G} \cap N$. We claim that $N$ has a normal Hall $\pi$-subgroup. Clearly we may assume $\rad{\pi}{N}=1$. Let $H\in\hall{\pi}{N}$, and let $p\in \pi$. If $x\in N\cap\van{G}$ is a $p$-element, then $x\in P\in\syl{p}{N}$. Hence there exists $g\in N$ such that $x^g \in P^g\in\syl{p}{H}$. Now Lemma \ref{wielandt} yields $x^g\in \rad{\pi}{N}=1$. Thus there are no $p$-elements in $N\cap \op{Van}(G)$, and by Theorem \ref{teoA} we get that $N$ has a normal Sylow $p$-subgroup. Since this is valid for every prime $p\in \pi$, then $N$ has a (nilpotent) normal Hall $\pi$-subgroup, as desired.

Next we show that $N$ has nilpotent Hall $\pi'$-subgroups under the additional assump\-tion that the prime power order $\pi'$-elements in $N\cap \van{G}$ have also $G$-class sizes not divisible by any prime in $\pi'$. Note that $N$ is $\pi$-separable because it has a normal Hall $\pi$-subgroup, say $H$. If we take any prime power order element $xH\in (N/H) \cap \van{G/H}$, then we may suppose that $x\in N\cap \van{G}$ is a prime power order element, so by assumptions $\abs{x^G}$ is a $\pi$-number. Thus $\abs{(xH)^{G/H}}$ is also a $\pi$-number. Therefore every $\abs{(xH)^{G/H}}$ is a $\pi$-number for each prime power order $\pi'$-element $xH\in (N/H) \cap \van{G/H}$, so by assertion (1) the $\pi'$-group $N/H$ is nilpotent. Since $N/H$ is isomorphic to a Hall $\pi'$-subgroup of $N$, the proof is completed.
\end{proof}

\begin{example}
\label{pi-separa}
We remark that the $\pi$-separability assumption in Theorem \ref{teoC} (1) is necessary for the first claim. Let $G$ be a symmetric group of degree $5$, and let $N$ be an alternating group of degree $5$. Consider $\pi=\{3\}$. Then all the $3$-elements in $N\cap \van{G}$ have conjugacy class size equal to $20$. Nevertheless, $N/\rad{\pi'}{N}=N$ does not have a normal Sylow $3$-subgroup.
\end{example}

\begin{example}
It is not difficult to find groups satisfying the assumptions of Theorem \ref{teoC}. For instance, let $G=A\Gamma(2^3)$ be an affine semilinear group of order $168$, and let $N$ be the Hall $3'$-subgroup of $G$. If we consider $\pi=\{7\}$, then the pair $(N, G)$ satisfies the hypotheses of Theorem \ref{teoC} (1). Concerning Theorem \ref{teoC} (2), if $\pi$ is any set of prime numbers, $G=\rad{\pi}{G}\times\rad{\pi'}{G}$ and $N=\rad{\pi}{G}$, then the pair $(N, G)$ certainly holds the hypotheses. 
\end{example}

The next theorem combines the arithmetical conditions of Theorem \ref{teoC} on the vani\-shing $G$-class sizes.

\begin{theorem}
\label{pi-pi'}
Let $N$ be a normal $\pi$-separable subgroup of a group $G$. Assume that $\abs{x^G}$ is either a $\pi$-number or a $\pi'$-number for every prime power order $\pi$-element $x \in \van{G} \cap N$. Then $N/\rad{\pi'}{N}$ has a normal Hall $\pi$-subgroup. Thus $N$ has $\pi$-length at most 1.
\end{theorem}

\begin{proof}
First, we claim that $O:=\radd{\pi}{N}$ contains a Sylow $p$-subgroup of $N$, for a prime $p\in \pi$. Let $x \in \van{G} \cap N$ be a $p$-element. If $\abs{x^G}$ is a $\pi$-number, then $x$ lies in $\rad{\pi}{N}$ because of Lemma \ref{wielandt}, so clearly $x\in O$. If $\abs{x^G}$ is a $\pi'$-number, then by Proposition \ref{in_fitting} we get $x\rad{\pi'}{N}\in \fit{N/\rad{\pi'}{N}}$, and again $x$ lies in $O$. It follows that $\overline{N}:=N/O$ contains no vanishing $p$-element of $G/O$, so $\overline{N}$ has a normal Sylow $p$-subgroup $\overline{P}$ in virtue of Theorem \ref{teoA}. Since $p \in \pi$ and clearly $\rad{\pi}{\overline{N}}=1$, thus $\overline{P}=1$.

Therefore $O$ contains a Sylow $p$-subgroup of $N$ for every $p\in\pi$, and thus $O/\rad{\pi'}{N}$ is a Hall $\pi$-subgroup of $N/\rad{\pi'}{N}$.
\end{proof}

\bigskip

The main theorem of \cite{BF} examines groups such that all their $\pi$-elements have prime power class sizes. The next result is a ``vanishing version'' of that theorem for prime power order elements and in the context of $G$-conjugacy classes. 

\bigskip

\begin{theorem}
\label{teoE}
Let $N$ be a normal subgroup of a group $G$. Assume that $\abs{x^G}$ is a prime power for each prime power order $\pi$-element $x \in N$ that is vanishing in $G$. Then $N/\rad{\pi'}{\fit{N}}$ has a normal Hall $\pi$-subgroup.

In particular, if $\pi$ is the set of prime divisors of $\abs{N}$, then $N/\fit{N}$ is nilpotent. 
\end{theorem}

\begin{proof}
We claim that $\overline{N}:=N/\fit{N}$ has a normal Hall $\pi$-subgroup, and therefore $N/\rad{\pi'}{\fit{N}}$ so does because $\fit{N}/\rad{\pi'}{\fit{N}}$ is a $\pi$-group. Arguing by contradiction, and in virtue of Proposition \ref{nilphall}, we may assume that $\overline{N}\cap \van{\overline{G}}$ contains a non-trivial $q$-element for some prime $q\in \pi$, say $\overline{x}$. Hence we may suppose that $x\in (N\cap \van{G})\smallsetminus \fit{N}$ is a $q$-element. By assumptions, we have that $\abs{x^G}$ is a power of some prime $p$. Observe that, since $x\notin\fit{N}$, then $q\neq p$ due to Lemma \ref{wielandt}. Now the last statement of Proposition \ref{in_fitting} yields $(\langle x^G \rangle)'\leqslant \rad{p}{N}\leqslant \fit{N}$, so $\overline{\langle x \rangle}$ is a subnormal nilpotent subgroup of $\overline{N}$. It follows that $\overline{x}\in\fit{\overline{N}}$, and as $\overline{x}$ is a $q$-element, then $\overline{x}\in \rad{q}{\overline{N}}$. Now $\abs{\overline{x}^{\overline{G}}}$ is a multiple of $q$ by Proposition \ref{brough}, and then $\abs{x^G}$ so is, a contradiction. 

Finally, if $\pi=\pi(N)$, then with a similar argument we deduce that there is no prime power order element in $N/\fit{N}$ vanishing in $G/\fit{N}$. Hence Theorem \ref{teoA} applies and $N/\fit{N}$ is nilpotent.
\end{proof}

\section{Some consequences on vanishing conjugacy classes}
\label{secCor}

New interesting contributions on the lengths of vanishing classes of a group $G$ emerge from Theorem \ref{teoC}, Theorem \ref{pi-pi'} and Theorem \ref{teoE} when $N=G$.

\begin{theorem}
\label{cor_van_pi'}
Let $G$ be a $\pi$-separable group. If $\abs{x^G}$ is a $\pi'$-number for every prime power order $\pi$-element $x \in \van{G}$, then $G/\rad{\pi'}{G}$ has a nilpotent normal Hall $\pi$-subgroup. Therefore, $G$ has nilpotent Hall $\pi$-subgroups, and its $\pi$-length is at most 1.
\end{theorem}

\begin{theorem}
\label{cor_van_pi}
Let $G$ be a finite group such that $\hall{\pi}{G}\neq \emptyset$. Assume $\abs{x^G}$ is a $\pi$-number for every prime power order $\pi$-element $x\in \op{Van}(G)$. Then $G$ has a normal Hall $\pi$-subgroup.

Further, if the prime power order $\pi'$-elements in $\van{G}$ have also class size a $\pi$-number, then the Hall $\pi'$-subgroups of $G$ are nilpotent.
\end{theorem}

\begin{theorem}
\label{corE}
Let $G$ be a group. Suppose that $\abs{x^G}$ is either a $\pi$-number or a $\pi'$-number for every prime power order $\pi$-element $x \in \van{G}$. Then $G/\rad{\pi'}{\fit{G}}$ has a normal Hall $\pi$-subgroup. In particular, $G$ has $\pi$-length at most 1.
\end{theorem}

\begin{proof}[Proof of Theorem \ref{cor_nilpotent}]
Arguing as in the proof of Theorem \ref{teoE} we can see that $G/\fit{G}$ has no prime power order vanishing elements. Thus, Theorem \ref{carac} applies and $G/\fit{G}$ is abelian, so $G'$ is nilpotent.
\end{proof}

\bigskip


\noindent \textbf{Acknowledgements:} This research has been carried out during a stay of the second author at the Instituto Universitario de Matem\'atica Pura y Aplicada (IUMPA-UPV) of the Universitat Polit\`ecnica de Val\`encia. He wishes to thank the members of the IUMPA for their hospitality. The authors would like to thank G. Navarro for useful conversations during the preparation of the paper.


\end{document}